\documentclass[a4paper,11pt,openany]{article}
\usepackage{amssymb, amsmath}
\usepackage[T1]{fontenc}
\usepackage[english]{babel}
\usepackage[cp1250]{inputenc}
\usepackage[dvips]{graphicx}
\usepackage{color}
\usepackage{geometry}
\usepackage{latexsym}
\usepackage{extarrows}
\usepackage{bm}
\usepackage{braket}
\usepackage{amsthm}
\usepackage{bigints}
\usepackage{dsfont}
\usepackage{fancyhdr}
\usepackage{hyperref}
\usepackage{enumitem}

\newgeometry{tmargin=3.5cm, bmargin=3.5cm, lmargin=2.25cm, rmargin=2.25cm}
\setlist{leftmargin=1.5cm}
\newcommand{\R}{\mathbb R}

\newcommand{\N}{\mathbb N}

\newcommand{\X}{\mathbb X}

\newcommand{\E}{\mathbb E}

\newtheorem{tw}{Theorem}[section]

\newtheorem{lm}{Lemma}[section]
\theoremstyle{definition}

\newtheorem*{ex}{Example}
\newtheorem*{rem}{Remark}
\begin{document}
\title{\textbf{Global diffeomorphism theorem applied to the solvability of discrete and continuous boundary value problems}}
\date{ }
\maketitle
\begin{minipage}{0.5 \textwidth}
\begin{center}
\textbf{Micha\l \ Be\l dzi{\' n}ski}\\
beldzinski.michal@outlook.com
\end{center}
\end{minipage}
\begin{minipage}{0.5 \textwidth}
\begin{center}
\textbf{Marek Galewski}\\
marek.galewski@p.lodz.pl
\end{center}
\end{minipage}\\
\begin{center}
Lodz University of Technology\\Institute of Mathematics\\W{\' o}lcza{\' n}ska 215, 90-234 Lodz, Poland
\end{center}
\begin{abstract}
Using the global diffeomorphism theorem we consider the existence of solutions to the following Dirichlet problem $\ddot{x}\left(t\right)=f\left( t,x\left( t\right) \right)+v(t) $, $x\left( 0\right)=x\left( 1\right) =0$, where $f:\left[ 0,1\right] \times\mathbb{R}\rightarrow\mathbb{R}$ is a jointly continuous function subject to some further growth conditions. Together with a given problem we consider the family of its discretizations and as a result we prove the existence of a non-spurious solution.
\end{abstract}
\pagenumbering{arabic}
\section{Introduction}
In this note we consider in $H_{0}^{1}(0,1)\cap H^{2}(0,1)$ solvability of the following Dirichlet problem 
\begin{equation}
\left\{ 
\begin{array}{l}
\ddot{x}\left( t\right) =f\left( t,x\left( t\right) \right) +v(t),\\ 
x\left( 0\right) =x\left( 1\right) =0,
\end{array}
\right.  \label{continuous_problem}
\end{equation}%
where $f:\left[ 0,1\right] \times \mathbb{R}\rightarrow \mathbb{R}$ is a jointly continuous function and $v\in $ $L^{2}(0,1)$, together with its standard discretization. Convergence of explicitly provided finite dimensional approximations is also undertaken. The idea of solving (\ref{continuous_problem}) is as follows. We investigate the classical solution operator $T$ given (pointwisely) a.e. on $[0,1]$ by 
\begin{equation}
\label{operator_T_definition}
(Tx)(\cdot ):=\ddot{x}(\cdot )-f(\cdot ,x(\cdot )),
\end{equation}
acting from $H_{0}^{1}(0,1)\cap H^{2}(0,1)$ to $L^{2}(0,1)$ for which we can prove that $T$ is a global diffeomorphism using the assumptions which we impose on the nonlinear term $f$ and with the aid of a global diffeomorphism theorem from \cite{idczak}. Then we will have not only the existence of a solution to (\ref{continuous_problem}) but it also would be unique and it would depend continuously on a parameters of the system. We would like to mention that in this work it is the first attempt to prove the existence of solutions to a second order problem by a global diffeomorphism theorem. This theorem has been previously applied to the solvability of first order integro-differential systems, see for example \cite{idczak}, \cite{majewski}. There is also a related research which allows for obtaining the existence of unique solutions to second order ODE contained in \cite{radulescu}.

Together with (\ref{continuous_problem}) we shall consider its discretization which we take
form \cite{kellypeterson}, see also \cite{gaines}. For $a$, $b$ such that $a<b<\infty $, $a\in \mathbb{N}\cup \{0\}$, $b\in \mathbb{N}$ we denote $\mathbb{N}(a,b)=\{a,a+1,...,b-1,b\}$. For a fixed $N\in \mathbb{N}$, $N\geqslant 2$, the non-linear difference equation with Dirichlet boundary conditions is given as follows 
\begin{equation}
\left\{ 
\begin{array}{l}
\Delta ^{2}x(k-1)=\frac{1}{N^{2}}f\left( \frac{k}{N},x(k)\right) +\frac{1}{N^{2}}v\left( \frac{k}{N}\right) , \\ 
x(0)=x(N)=0,
\end{array}
\right.  \label{discretization}
\end{equation}
for $k\in \mathbb{N}(1,N-1)$. Here $\Delta $ is the forward difference operator, i.e. $\Delta x\left(k-1\right) =x\left( k\right) -x\left( k-1\right) $ and we see that $\Delta^{2}x\left( k-1\right) =x\left( k+1\right) -2x\left( k\right) +x\left(k-1\right) $. Assume that both, continuous boundary value problem (\ref{continuous_problem}) and for each fixed $N\in \mathbb{N}$, $N\geqslant 2$, discrete boundary value problem (\ref{discretization}), are uniquely solvable by, respectively $x^{\star }$ and $x_{N}=\left( x_{N}(k)\right) _{k=0}^{N}$.  Then, if $v$ is at least continuous, solutions $x_{N}$ of (\ref{discretization}) converges to solution $x^{\star }$ of (\ref{continuous_problem}) in following sense
\begin{equation}
\lim\limits_{N\to	\infty}\max_{k\in\N(0,N)}\left\vert x^{\star}\left( \tfrac{k}{N}\right) -x_{N}(k)\right\vert =0 \label{convergence_of_discretizations}
\end{equation}
Such solutions to discrete BVPs are called non-spurious. The spurious solutions may diverge or else may converge to anything else but the solution to a given continuous Dirichlet problem, see comments in \cite{gal}. We also refer to \cite{gal} for some examples relating the solvability of both continuous problem and its discrete counterpart. The definition of a non-spurious solution which we employ follows paper \cite{rachunkowa2} and is given as in \cite{gal}. The existence of a non-spurious solutions have been considered by variational methods in \cite{gal} while previously there had been some research in this case addressing mainly problems whose solutions where obtained by the fixed point theorems and the method of lower and upper solutions, \cite{rech1}, \cite{thomsontisdell}.

As in \cite{gal} variational methods are used but now the action functional which is considered differs substantially from the action functional commonly used for variational problems, see \cite{mawhin}, considered by a direct method. When compared with the usage of a direct method, we have to make the following comment. While the classical action functional connected with our problem with the assumptions which we are going to impose, would be coercive, it would not be strictly convex. This means that the solution obtained by using a direct method might not be unique. Instead of convexity we put some restrictions on the derivative of $f $ with respect to $x$. Again this is not an uncommon situation for the term connected with the nonlinearity to be differentiable.

The paper is organized as follows. Firstly we recall some preliminaries which we will need in our main result. Then we prove that our continuous problem (\ref{continuous_problem}) has a unique solution via global diffeomorphism theorem. Next we show unique solvability of discretizations. Finally, we investigate the convergence of solutions of (\ref{discretization}) to the solution of (\ref{continuous_problem}) in a sense described in (\ref{convergence_of_discretizations}). 
\section{Preliminaries}
Let $X$ and $Y$ be a Banach spaces. We say that functional $f:X\rightarrow \mathbb{R}$ is coercive if $\lim\limits_{\Vert x\Vert_X \rightarrow \infty }f(x)=\infty$. A mapping $f:X\to Y$ is said to be diffeomorphism if $f$ is bijective, $C^{1}$ and its inverse $f^{-1}$ is also $C^{1}$. Observe that if function $f$ has a directional derivative along every vector at point $x_{0}$ and $f'(x_0)=f'(x_{0};\cdot)\in \mathcal{L}(X,Y)$ then $f$ is G{\^ a}teaux differentiable at $x_{0}$. Moreover, if the G{\^ a}teaux differential exists on some open boundary containing $x_{0}$ and operator $X\ni x\mapsto f'(x)\in\mathcal{L}(X,Y)$ is continuous at $x_0$, then it is in fact Fr{\' e}chet-derivative of $f$ at point $x_{0}$. As a consequence the existence of G{\^ a}teaux differential of $f$ and its continuity at every point of $X$ implies that $f$ is $C^{1}$.

Let us present main tools of this paper. We first consider finite-dimensional case.
\begin{tw}[Hadamard, \cite{jabri}]
\label{theorem_Hadamard} Let $X$ and $Y$ be a finite-dimensional Hilbert spaces. Assume that $F:X\rightarrow Y$ is $C^{1}$ mapping satisfying following conditions
\begin{itemize}
\item $F^{\prime}(x)$ is invertible at every point $x\in X$,
\item functional $x\mapsto \Vert F(x)\Vert _{Y}$ is coercive.
\end{itemize}
Then $F$ is diffeomorphism.
\end{tw}
We say that a $C^{1}$ functional $\varphi :X\rightarrow \mathbb{R}$, where $X $ is a Banach space, satisfies \textbf{the Palais-Smale (PS) condition} if every sequence $(x_{n})_{n\in \mathbb{N}}$ such that $(\varphi(x_{n}))_{n\in \mathbb{N}}$ is bounded and $\varphi ^{\prime }(x_{n})\xrightarrow{X^*}\theta_{X^{\ast }}$, admits a convergent subsequence.
\begin{tw}[Idczak-Skowron-Walczak, \cite{idczak}]
\label{theorem_Idczak} Let $X$ be a real Banach space, $H$ - a real Hilbert space. If $F:X\rightarrow H$ is a $C^{1}$ mapping such that
\begin{itemize}
\item[(d1)] for every $y\in H$ functional $\varphi:X\to\mathbb{R}$ given by $\varphi(x):=\tfrac{1}{2}\|F(x)-y\|_H^2$ satisfies the PS condition,
\item[(d2)] for any $x\in X$ an operator $F^{\prime }(x)$ is invertible and $F^{\prime }(x)[X]=H$,
\end{itemize}
then $F$ is diffeomorphism.
\end{tw}
\section{Solvability of the continuous boundary value problem}
We start this section with some remarks on the space in which we consider our problem. Firstly, denote $L^2:=L^2(0,1)$. The space $H^{1}$ consists of those absolute continuous functions, defined on closed interval $[0,1]$, whose weak derivative is integrable with square. $H^{2}$ denotes space of those functions form $H^1$ for which $\dot{x}\in H^1$. We define $H_{0}^{1}:=\{x\in H^{1}:x(0)=x(1)=0\}$. Let $x\in H_{0}^{1}$. The following inequalities hold, see \cite{haase},
\begin{equation*}
\begin{array}{cc}
\Vert x\Vert _{\infty }\leqslant \Vert \dot{x}\Vert _{L^{2}},\hspace{2.5cm} & \Vert x\Vert
_{L^{2}}\leqslant \frac{1}{\pi }\Vert \dot{x}\Vert _{L^{2}}.
\end{array}
\end{equation*}
Hence we can consider the spaces mentioned above with the following norms\\
\centerline{$\Vert x\Vert _{H^{2}}:=\Vert x\Vert _{L^{2}}+\Vert \dot{x}\Vert_{L^{2}}+\Vert \ddot{x}\Vert _{L^{2}}$, \ \ \ \ $\Vert x\Vert _{H^{1}}:=\Vert x\Vert_{L^{2}}+\Vert \dot{x}\Vert _{L^{2}}$ \ \ and \ \ $\Vert x\Vert _{H_{0}^{1}}:=\Vert\dot{x}\Vert _{L^{2}}$.}

Finally we define space $\mathbb{X}:=H^{2}\cap H_{0}^{1}$ equipped with a norm $\Vert x\Vert _{\mathbb{X}}:=\Vert \ddot{x}\Vert _{L^{2}}$.
\begin{lm}
$\mathbb{X}$ is a closed subspace of $H^{2}$ and norms $\Vert \cdot \Vert _{\mathbb{X}}$ and $\Vert \cdot \Vert _{H^{2}}$ are equivalent on $\X$. Moreover, for every $x\in\X$, the following inequalities hold 
\begin{equation}
\Vert x\Vert _{\infty }\leqslant \Vert \dot{x}\Vert _{L^{2}},\hspace{2.5cm}
\Vert x\Vert _{L^{2}}\leqslant \tfrac{1}{\pi }\Vert \dot{x}\Vert
_{L^{2}}\leqslant \tfrac{1}{\pi ^{2}}\Vert \ddot{x}\Vert _{L^{2}}. \label{ineq_from_assert}
\end{equation}
\end{lm}
\begin{proof}
Suppose $(x_{n})_{n\in \mathbb{N}}\subset \mathbb{X}$ and $x_{n}\xrightarrow{H^2}x^{\star }$. Then, by Lemma 1.2 from \cite{mawhin}, we have $x_{n}\xrightarrow{C[0,1]}x^{\star }$. That implies the pointwise convergence of $(x_{n})_{n\in \mathbb{N}}$ on $[0,1]$. Hence $x^{\star}(0)=\lim\limits_{n\rightarrow \infty }x_{n}(0)=0$. For the same reason $x^{\star }(1)=0$.

Let $x\in \mathbb{X}$. Then $x$ and $\dot{x}$ are absolutely continuous and $x(0)=x(1)=0$. Integrating by parts we see that 
\begin{equation*}
\Vert x\Vert _{L^{2}}\Vert \dot{x}\Vert _{L^{2}}\leqslant \tfrac{1}{\pi }\Vert \dot{x}\Vert _{L^{2}}^{2}=\tfrac{1}{\pi }\int_{0}^{1}\dot{x}(t)\dot{x}(t)dt=-\tfrac{1}{\pi }\int_{0}^{1}x(t)\ddot{x}(t)dt\leqslant \tfrac{1}{\pi }\Vert x\Vert _{L^{2}}\Vert \ddot{x}\Vert _{L^{2}}.
\end{equation*}
As a result we have (\ref{ineq_from_assert}). Thus both norms are equivalent
on $\mathbb{X}$.
\end{proof}
The solutions to (\ref{continuous_problem}) will be investigated in the space $\mathbb{X}$. Such a solution we call classical. In case $v$ is at least continuous then the solution belongs to $C^{2}[0,1]$. The following theorem shows when the problem (\ref{continuous_problem}) has exactly one solution. To make notation shorter, let us put $\inf\limits_{\Omega }f:=\inf\limits_{\omega\in \Omega }f(\omega)$.
\begin{tw}
\label{theorem_solvability_of_continuous} Let $v\in L^2$. Suppose that $f:[0,1]\times\mathbb{R}\to\mathbb{R}$ satisfies the following assumptions
\begin{itemize}
\item[\textbf{C(c)}:] $f$ is continuous on $[0,1]\times\R$ and it has continuous partial derivative $f_{x}$ with respect to the second coordinate on $[0,1]\times\R$, 
\item[\textbf{C($f$)}:] there exist a positive constants $A$ and $B$, $A<\pi^2$, such that for every $t\in[0,1]$ there is $|f(t,x)|\leqslant A|x|+B$,
\item[\textbf{C($f_x$)}:] $\inf\limits_{[0,1]\times\mathbb{R}} f_{x}> -\pi^2$.
\end{itemize}
Then problem (\ref{continuous_problem}) has a unique solution.
\end{tw}
In order to prove Theorem \ref{theorem_solvability_of_continuous} it is enough to show that $T$ given by (\ref{operator_T_definition}) is diffeomorphism. This in turn will be done with the aid of Theorem \ref{theorem_Idczak}. In order to apply this theorem we fix $y\in L^{2}$ and we define functional $\varphi :\mathbb{X}\rightarrow \mathbb{R}$ by 
\begin{equation*}
\varphi (x):=\tfrac{1}{2}\Vert Tx-y\Vert _{L^{2}}^{2}=\tfrac{1}{2}\int_{0}^{1}|\ddot{x}(t)-f(t,x(t))-y(t)|^{2}dt.
\end{equation*}
We will verify that conditions (d1) and (d2) of Theorem \ref{theorem_Idczak} hold provided \textbf{\textit{C(c)}}, \textbf{\textit{C($f$)}} and \textbf{\textit{C($f_{x} $)}} are imposed.
\subsection{Proof of the main result}
We begin with showing that condition (d1) is satisfied. Firstly we need to prove that $T$ is a $C^{1}$ mapping. Then in the sequence of lemmas we will investigate the properties of functional $\varphi $. We start with showing that it is $C^{1}$, then that it is coercive and finally that it satisfies the PS condition.
\begin{lm}
\label{lemma_T'} Assume that \textbf{\textit{C(c)}} is satisfied. Then operator $T:\mathbb{X}\rightarrow L^{2}$ is a $C^{1}$ mapping with a derivative at any fixed $x\in \mathbb{X}$ given by the following formula (for a.e. $t\in \left[ 0,1\right] $) 
\begin{equation*}
T^{\prime }(x;\psi )(t)=\ddot{\psi}(t)-f_{x}(t,x(t))\psi (t)
\end{equation*}
for all $\psi \in \mathbb{X}$.
\end{lm}
\begin{proof}
Let $x,\psi \in \mathbb{X}$ be fixed. Then (for a.e. $t\in \left[ 0,1\right]$) by the assumptions on $f$ we get 
\begin{align*}
T^{\prime }(x;\psi )(t)& =\left. \dfrac{d}{d\tau }T(x+\tau \psi )\right\vert_{\tau =0}=\left. \dfrac{d}{d\tau }\ddot{x}(t)+\tau\ddot{\psi}(t)-f(t,x(t)+\tau \psi (t))\right\vert _{\tau =0} \\
& =\ddot{\psi}(t)-f_{x}(t,x(t))\psi (t).
\end{align*}
We show that $T^{\prime }(x;\cdot )$ is bounded and thus continuous. Because $f_{x}$ and $x$ are continuous then $|f_x(t,x(t))|$ is bounded by some positive constant $M$ (for a.e. $t\in[0,1]$). Then 
\begin{align*}
\Vert T^{\prime }(x;\psi )\Vert _{L^{2}}& =\left( \int_{0}^{1}|\ddot{\psi}(t)-f_{x}(t,x(t))\psi (t)|^{2}dt\right) ^{\frac{1}{2}} \\ 
&\leqslant \left( \int_{0}^{1}|\ddot{\psi}(t)|^{2}dt\right) ^{\frac{1}{2}}+M\left( \int_{0}^{1}|\psi (t)|^{2}dt\right) ^{\frac{1}{2}}\leqslant \left(1+\tfrac{M}{\pi ^{2}}\right) \Vert \psi \Vert _{\mathbb{X}}.
\end{align*}%
Hence $T^{\prime }(x)$ is a G{\^ a}teaux-differential of $T$. The last step is to show continuity of this operator at every $x\in \mathbb{X}$. Fix $x_{0}\in \mathbb{X}$ and let $x_{n}\xrightarrow{\X}x_{0}$, then $x_{n}\xrightarrow{C[0,1]}x_{0}$ and consequently, since $f_{x}$ is continuous, $f_{x}(t,x_{n}(t))\rightarrow f_{x}(t,x_{0}(t))$ (for a.e. $t\in\lbrack 0,1]$). Moreover, $|f_{x}(t,x_{n}(t))|$ is uniformly bounded (for a.e. $t\in[0,1]$), since $(x_{n})_{n\in \mathbb{N}}$ is bounded in $C[0,1]$. Hence, by the Lebesgue Dominated Convergence Theorem we obtain 
\begin{align*}
\Vert T^{\prime }(x_{0} )-T^{\prime }(x_{n})\Vert _{\mathcal{L}\left( \mathbb{X},L^{2}\right) }& =\sup\limits_{\Vert \xi \Vert _{\mathbb{X}}=1}\Vert T^{\prime }(x_{0};\xi )-T^{\prime }(x_{n};\xi )\Vert _{L^{2}} 
\\& \leqslant \sup\limits_{\Vert \xi \Vert _{\mathbb{X}}=1}\Vert \xi \Vert_{\infty }\left(\int_{0}^{1}|f_{x}(t,x_{n}(t))-f_{x}(t,x_{0}(t))|^{2}dt\right) ^{\frac{1}{2}}\rightarrow 0.
\end{align*}
Finally, $T\in C^{1}\left( \mathbb{X},L^{2}\right) $.
\end{proof}
By Lemma \ref{lemma_T'} and by the Chain Rule formula, we obtain the following
\begin{lm}
\label{lemma_varphi_is_C^1} Assume that \textbf{\textit{C(c)}} is satisfied. Then the functional $\varphi :\mathbb{X}\rightarrow \mathbb{R}$ is $C^{1}$ with a derivative given at any fixed $x\in \mathbb{X}$ by the following formula 
\begin{equation*}
\varphi^{\prime }(x;h)=\int_{0}^{1}\left( \ddot{x}(t)-f(t,x(t))-y(t)\right)\left( \ddot{h}(t)-f_{x}(t,x(t))h(t)\right) dt
\end{equation*}
for all $h\in \mathbb{X}$.
\end{lm}
\begin{lm}
\label{lemma_varphi_is_coercive} Assume that and \textbf{\textit{C($f$)}} are satisfied. Then the functional $\varphi :\mathbb{X}\rightarrow \mathbb{R}$ is coercive.
\end{lm}
\begin{proof}
Fix $x\in\X$. Using assumption we obtain
\begin{align*}
\|Tx-y\|_{L^2}&=\left(\int_0^1 |\ddot{x}(t)-f(t,x(t))-y(t)|^{2}dt\right)^\frac{1}{2} \\
&\geqslant\left(\int_{0}^{1}|\ddot{x}(t)|^{2}dt\right)^\frac{1}{2}-\left(\int_{0}^{1}|f(t,x(t))|^{2}dt\right)^\frac{1}{2}-\left(\int_{0}^{1}|y(t)|^{2}dt\right)^\frac{1}{2} \\
&\geqslant\left(\int_{0}^{1}|\ddot{x}(t)|^{2}dt\right)^\frac{1}{2}-\left(\int_{0}^{1}|Ax(t)+B|^{2}dt\right)^\frac{1}{2}-\left(\int_{0}^{1}|y(t)|^{2}dt\right)^\frac{1}{2} \\
&\geqslant\|\ddot{x}\|_{L^2}-A\|x\|_{L^2}-B-\|y\|_{L^2}\geqslant\left( 1-\tfrac{A}{\pi ^{2}}\right) \Vert x\Vert _{\mathbb{X}}-B-\Vert y\Vert _{L^{2}}.
\end{align*}
Hence $\varphi $ is coercive.
\end{proof}
\begin{lm}
\label{lamma_varphi_is_PS} Assume that \textbf{\textit{C(c)}} and \textbf{\textit{C($f$)}} are satisfied. Then functional $\varphi :\mathbb{X}\rightarrow\mathbb{R}$ satisfies the PS condition.
\end{lm}
\begin{proof}
Let $(x_{n})_{n\in \mathbb{N}}\subset \mathbb{X}$ be a PS sequence
for $\varphi $, i.e.\\
\begin{minipage}{0.5\textwidth}
\begin{itemize}
\item[(PS1)] $(\varphi(x_n))_{n\in\mathbb{N}}$ is bounded,
\end{itemize}
\end{minipage}
\begin{minipage}{0.5\textwidth}
\begin{itemize}
\item[(PS2)] $\varphi^{\prime }(x_n)\xrightarrow{\X^*} \theta_{\mathbb{X}^*}$.
\end{itemize}
\end{minipage}
By Lemmas \ref{lemma_varphi_is_C^1} and \ref{lemma_varphi_is_coercive} functional $\varphi $ is coercive and $C^{1}$. Hence sequence $(x_n)_{n\in\N}$ is bounded in $\X$ by (PS1). Therefore, since $\mathbb{X}$ is reflexive, there exists a subsequence $\left( x_{n_{k}}\right) _{k\in \mathbb{N}}$ and point $x_{0}$ such that $x_{n_{k}}\overset{\mathbb{X}}{\rightharpoonup }x_{0}$. Without loss of the generality, we may assume that $x_{n}\overset{\mathbb{X}}{\rightharpoonup }x_{0}$. Therefore by definition of $\mathbb{X}$ it follows that $x_{n}\overset{H_{0}^{1}}{\rightharpoonup }x_{0}$ and $x_{n}\overset{H^{2}}{\rightharpoonup }x_{0}$. Hence $x_{n}\xrightarrow{C\left[ 0,1\right] }x_{0}$ and $x_{n}\xrightarrow{L^{2}}x_{0}$. It could be calculated that 
\begin{equation}
\varphi ^{\prime }(x_{0})(x_{0}-x_{n})-\varphi ^{\prime}(x_{n})(x_{0}-x_{n})=\Vert \ddot{x}_{0}-\ddot{x}_{n}\Vert_{L^{2}}^{2}+\sum_{i=1}^{6}\psi _{i}(x_{n}),  \label{PS}
\end{equation}
where 
\begin{align*}
\psi _{1}(x_{n})& =-\int_{0}^{1}(\ddot{x}_{n}(t)-\ddot{x}_{0}(t))(f(t,x_{n}(t))-f(t,x_{0}(t)))dt, \\
\psi _{2}(x_{n})& =-\int_{0}^{1}\ddot{x}_{n}(t)f_{x}(t,x_{n}(t))(x_{n}(t)-x_{0}(t))dt, \\
\psi _{3}(x_{n})& =\int_{0}^{1}\ddot{x}_{0}(t)f_{x}(t,x_{0}(t))(x_{n}(t)-x_{0}(t))dt, \\
\psi _{4}(x_{n})&=\int_{0}^{1}y(t)(f_{x}(t,x_{n}(t))-f_{x}(t,x_{0}(t)))(x_{n}(t)-x_{0}(t))dt,\\
\psi _{5}(x_{n})&=\int_{0}^{1}f(t,x_{n}(t))f_{x}(t,x_{n}(t))(x_{n}(t)-x_{0}(t))dt, \\
\psi _{6}(x_{n})&=-\int_{0}^{1}f(t,x_{0}(t))f_{x}(t,x_{0}(t))(x_{n}(t)-x_{0}(t))dt.
\end{align*}
Using the Lebesgue Dominated Convergence Theorem and assumption \textbf{\textit{C(c)}} we obtain that $f(\cdot,x_{n}(\cdot ))\xrightarrow{L^2}f(\cdot ,x_{0}(\cdot ))$ and $f_{x}(\cdot,x_{n}(\cdot ))\xrightarrow{L^2}f_{x}(\cdot ,x_{0}(\cdot ))$. Hence 
\begin{align*}
|\psi _{1}(x_{n})|&\leqslant \left(\int_{0}^{1}|f_{x}(t,x_{n}(t))-f_{x}(t,x_{0}(t))|^{2}dt\right) ^{\frac{1}{2}}\left( \int_{0}^{1}|\ddot{x}_{0}(t)-\ddot{x}_{n}(t)|^{2}dt\right) ^{\frac{1}{2}}\rightarrow 0,\\
|\psi _{2}(x_{n})|& \leqslant \Vert x_{n}-x_{0}\Vert _{\infty }\int_{0}^{1}|\ddot{x}_{n}(t)f_{x}(t,x_{n}(t))|dt \\
& \leqslant \Vert x_{n}-x_{0}\Vert _{\infty }\Vert \ddot{x}_{n}\Vert_{L^{2}}\left( \int_{0}^{1}|f_{x}(t,x_{n}(t))|^{2}dt\right) ^{\frac{1}{2}}\rightarrow 0,\\
|\psi _{3}(x_{n})|&\leqslant \Vert x_{n}-x_{0}\Vert _{\infty }\int_{0}^{1}|\ddot{x}_{0}(t)f_{x}(t,x_{0}(t))|dt\rightarrow 0,\\
|\psi _{4}(x_{n})|&\leqslant \Vert x_{n}-x_{0}\Vert _{\infty }\Vert f_{x}(\cdot ,x_{n}(\cdot ))-f_{x}(\cdot ,x_{0}(\cdot ))\Vert _{\infty
}\int_{0}^{1}|y(t)|dt\rightarrow 0,\\
|\psi _{5}(x_{n})|&\leqslant \Vert f(\cdot ,x_{n}(\cdot ))\Vert _{\infty}\cdot \Vert f_{x}(\cdot ,x_{n}(\cdot ))\Vert _{\infty }\cdot \Vert x_{n}-x_{0}\Vert _{\infty }\rightarrow 0,\\
|\psi _{6}(x_{n})|&\leqslant \Vert f(\cdot ,x_{0}(\cdot ))\Vert _{\infty}\Vert f_{x}(\cdot ,x_{0}(\cdot ))\Vert _{\infty }\cdot \Vert
x_{n}-x_{0}\Vert _{\infty }\rightarrow 0.
\end{align*}
Moreover, since $(x_{n})_{n\in \mathbb{N}}$ is weakly convergent, we get $\varphi ^{\prime }(x_{0})(x_{0}-x_{n})\rightarrow 0$. By boundedness of $(x_{n})_{n\in \mathbb{N}}$ in $\mathbb{X}$ and by (PS2) we have 
\begin{equation*}
|\varphi ^{\prime }(x_{n})(x_{0}-x_{n})|\leqslant \Vert \varphi ^{\prime}(x_{n})\Vert _{\mathbb{X}^{\ast }}\Vert x_{0}-x_{n}\Vert _{\mathbb{X}}\rightarrow 0.
\end{equation*}
Equality (\ref{PS}) implies $\Vert \ddot{x}_{n}-\ddot{x}_{0}\Vert_{L^{2}}\rightarrow 0$, which means that $\varphi $ satisfies the PS condition.
\end{proof}
Now we show that condition (d2) of Theorem \ref{theorem_Idczak} is satisfied. To do this we prove that, for every fixed $x\in \mathbb{X}$ and $y\in L^{2}$, the following problem 
\begin{equation}
\left\{ 
\begin{array}{l}
\ddot{\xi}(t)=f_{x}(t,x(t))\xi (t)+y(t), \\ 
\xi (0)=\xi (1)=0,
\end{array}
\right.  \label{invertible}
\end{equation}
has unique solution in $\mathbb{X}$. We define the Euler action functional $J:H_{0}^{1}\rightarrow \mathbb{R}$ by the formula 
\begin{equation*}
J(\xi )=\tfrac{1}{2}\int_{0}^{1}|\dot{\xi}(t)|^{2}dt+\tfrac{1}{2}\int_{0}^{1}f_{x}(t,x(t))|\xi (t)|^{2}dt+\int_{0}^{1}y(t)\xi (t)dt.
\end{equation*}
\begin{lm}
\label{J_is_coercive_l.s.c.} Assume that \textbf{\textit{C(c)}} and \textbf{\textit{C($f_x$)}} are satisfied. Then the functional $J$ is of class $C^{1}$, it is strictly convex and coercive on $H_{0}^{1}$.
\end{lm}
\begin{proof}
Fix $x\in\X$, then $f_{x}(\cdot ,x(\cdot ))$ is a fixed continuous function by \textbf{\textit{C(c)}} and hence $J$ is $C^{1}$. Using \textbf{\textit{C($f_{x}$)}} we get the existence of some negative constant $B$, $B>-\pi ^{2}$, such that $\inf\limits_{[0,1]\times \mathbb{R}}f_{x}\geqslant B$. Hence $J$ is coercive. Indeed, 
\begin{align*}
J(\xi )& \geqslant \tfrac{1}{2}\int_{0}^{1}|\dot{\xi}(t)|^{2}dt+\tfrac{1}{2}\inf\limits_{[0,1]\times \mathbb{R}}f_{x}\int_{0}^{1}|\xi
(t)|^{2}dt-\left\vert \int_{0}^{1}y(t)\xi (t)dt\right\vert \\
& \geqslant \tfrac{1}{2}\int_{0}^{1}|\dot{\xi}(t)|^{2}dt+\tfrac{B}{2\pi ^{2}}\int_{0}^{1}|\dot{\xi}(t)|^{2}dt-\left( \int_{0}^{1}|y(t)|^{2}dt\right) ^{\frac{1}{2}}\left( \int_{0}^{1}|\xi (t)|^{2}dt\right) ^{\frac{1}{2}} \\
& \geqslant\tfrac{\pi^2+B}{2\pi^2}\Vert \xi \Vert _{H_{0}^{1}}^{2}-\tfrac{1}{\pi }\Vert y\Vert _{L^{2}}\Vert\xi \Vert _{H_{0}^{1}}.
\end{align*}
We show that $J$ is strictly convex. Because a sum of convex and strictly convex functions is strictly convex then it is enough to proof that $\displaystyle J_{1}(u)=\int_{0}^{1}|\dot{u}(t)|^{2}+f_{x}(t,x(t))|u(t)|^{2}dt $ is strictly convex. Fix $\lambda \in
(0,1)$ and $u,w\in H_{0}^{1}$ such that $u\neq w$. Then 
\begin{align*}
(1-\lambda )J_{1}(u)& +\lambda J_{1}(w)-J_{1}((1-\lambda )u+\lambda w) \\
& =\lambda (1-\lambda )\int_{0}^{1}|\dot{u}(t)-\dot{w}(t)|^{2}+f_{x}(t,x(t))|u(t)-w(t)|^{2}dt \\
& \geqslant \lambda (1-\lambda )\left( \int_{0}^{1}|\dot{u}(t)-\dot{w}(t)|^{2}dt+\inf\limits_{[0,1]\times \mathbb{R}}f_{x}%
\int_{0}^{1}|u(t)-w(t)|^{2}dt\right) \\
& \geqslant \lambda (1-\lambda )\left( \pi ^{2}\Vert u-w\Vert_{L^{2}}^{2}+B\Vert u-w\Vert _{L^{2}}^{2}\right) \geqslant \lambda (1-\lambda )\left( \pi ^{2}+B\right) \Vert u-w\Vert_{L^{2}}^{2}>0.
\end{align*}
Therefore $(1-\lambda )J(u)+\lambda J(w)>J((1-\lambda )u+\lambda w)$ and the assertion holds.
\end{proof}
\begin{lm}
\label{lemma_invertible} Assume that \textbf{\textit{C(c)}} and \textbf{\textit{C($f_x$)}} are satisfied. Then problem (\ref{invertible}) has exactly one solution in $\mathbb{X}$.
\end{lm}
\begin{proof}
By Lemma \ref{J_is_coercive_l.s.c.} $J$ is coercive and strictly convex. Since $J$ is continuous and strictly convex it is also weakly lower semicontinuous. Hence it possesses exactly one critical point. It means that there exists exactly one $\xi ^{\star }\in H_{0}^{1}$ such that 
\begin{equation*}
\int_{0}^{1}\dot{\xi}^{\star }(t)\dot{\psi}(t)dt=-\int_{0}^{1}\left(
f_{x}(t,x(t))\xi ^{\star }(t)+y(t)\right) \psi (t)dt
\end{equation*}
for every $\psi \in H_{0}^{1}$. From the du Bois-Reymond Lemma, Lemma 1.1. from \cite{mawhin}, there exist a number $c\in\mathbb{R}$ such that 
\begin{equation*}
\dot{\xi}^{\star }(t)=\int_{0}^{t}f_{x}(s,x(s))\xi ^{\star }(s)+y(s)ds+c
\end{equation*}
for almost every $t\in \lbrack 0,1]$. Since $f_{x}(\cdot ,x(\cdot ))\xi^{\star }(\cdot )+y(\cdot )$ is integrable with square and hence $\xi^{\star }\in\mathbb{X}$.
\end{proof}
We conclude our considerations with following theorem
\begin{tw}
\label{theorem_T_is_diffeomorphism} Assume that \textbf{\textit{C(c)}}, \textbf{\textit{C($f$)}} and \textbf{\textit{C($f_x$)}} are satisfied. Then operator $T$ is diffeomorhism.
\end{tw}
\begin{proof}
It is a consequence of Lemmas \ref{lemma_T'}, \ref{lamma_varphi_is_PS}, \ref{lemma_invertible} and Theorem \ref{theorem_Idczak}.
\end{proof}
As we mentioned before, Theorem \ref{theorem_T_is_diffeomorphism} proves Theorem \ref{theorem_solvability_of_continuous}. 
\section{Solvability of the discrete boundary value problem}
Fix $N\in\N$, $N\geqslant 2$. We define $\mathbb{E}_{N}$ as a space of those functions $x:\mathbb{N}(0,N)\rightarrow \mathbb{R}$ for which $x(0)=x(N)=0$. Let us put $x(k)=0$ for every $k\notin \mathbb{N}(0,N)$. Consider the following norms on $\mathbb{E}_{N}$
$$
\begin{array}{ll}
\displaystyle\Vert x\Vert _{N}:=\left( \sum_{i=1}^{N-1}|x(i)|^{2}\right) ^{\frac{1}{2}}, &\displaystyle\Vert x\Vert _{\Delta _{N}}:=\left( \sum_{i=1}^{N}|\Delta x(i-1)|^{2}\right) ^{\frac{1}{2}},\\
\displaystyle\Vert x\Vert _{\mathbb{E}_{N}}:=\left( \sum_{i=1}^{N-1}|\Delta ^{2}x(i-1)|^{2}\right) ^{\frac{1}{2}}, &\displaystyle\Vert x\Vert _{\infty _{N}}:=\max\limits_{k\in \mathbb{N}(0,N)}|x(k)|.
\end{array}
$$
\begin{lm}
\label{lemma_inequalieties} For every $x\in \mathbb{E}_{N}$ we have 
\begin{equation*}
\tfrac{1}{4}\Vert x\Vert _{\mathbb{E}_{N}}\leqslant \tfrac{1}{2}\Vert x\Vert_{\Delta _{N}}\leqslant \Vert x\Vert _{N}\leqslant \sqrt{N}\Vert x\Vert_{\infty _{N}}\leqslant N\Vert x\Vert _{\Delta _{N}}\leqslant N^{2}\Vert x\Vert _{\mathbb{E}_{N}}.
\end{equation*}
\end{lm}
\begin{proof}
Fix $x\in \mathbb{E}$. Using Minkowski inequality we obtain 
\begin{equation*}
\Vert x\Vert _{\Delta _{N}}=\left( \sum_{i=1}^{N}|x(i)-x(i-1)|^{2}\right) ^{\frac{1}{2}}\leqslant \left( \sum_{i=1}^{N}|x(i)|^{2}\right) ^{\frac{1}{2}}+\left( \sum_{i=1}^{N}|x(i-1)|^{2}\right) ^{\frac{1}{2}}=2\Vert x\Vert _{\mathbb{E}_{N}}.
\end{equation*}
Now, since for every $k\in \mathbb{N}(0,N)$ it holds $|x(k)|\leqslant \Vert x\Vert _{\infty _{N}}$, we get 
\begin{equation*}
\Vert x\Vert _{N}=\left( \sum_{i=1}^{N-1}|x(i)|^{2}\right) ^{\frac{1}{2}}\leqslant \Vert x\Vert _{\infty _{N}}\left( \sum_{i=1}^{N-1}1^{2}\right) ^{\frac{1}{2}}\leqslant \sqrt{N}\Vert x\Vert _{\infty _{N}}.
\end{equation*}
The maximum of $x$ is achieved at some $k^{\star }\in \mathbb{N}(0,N)$. Hence by H{\" o}lder's inequality
\begin{align*}
\Vert x\Vert _{\infty _{N}}& =|x(k^{\star })|=\left\vert\sum_{i=1}^{k^{\star }}\Delta x(i-1)\right\vert \leqslant\sum_{i=1}^{N}\left\vert \Delta x(i-1)\right\vert \\
& \leqslant \left( \sum_{i=1}^{N}\left\vert \Delta x(i-1)\right\vert^{2}\right) ^{\frac{1}{2}}\left( \sum_{i=1}^{N}1^{2}\right) ^{\frac{1}{2}}=\sqrt{N}\Vert x\Vert _{\Delta _{N}}.
\end{align*}
We already showed that $\tfrac{1}{2}\Vert x\Vert_{\Delta_{N}}\leqslant\Vert x\Vert_{N}\leqslant \sqrt{N}\Vert x\Vert_{\infty_{N}}\leqslant N\Vert x\Vert_{\Delta_{N}}$. To prove the two remaining inequalieties we need to recall summation by parts formula from \cite{kellypeterson}. If $(a_{n})_{n\in \mathbb{N}}$, $(b_{n})_{n\in \mathbb{N}}$ are real sequences and  $m\in \mathbb{N}$ is a fixed integer then 
\begin{equation}
\sum_{k=1}^{m}a_{k}\Delta b_{k}=a_{m+1}b_{m+1}-a_{1}b_{1}-\sum_{k=1}^{m}\Delta a_{k}b_{k+1}.
\label{summation_by_parts}
\end{equation}
Fixing $x\in \mathbb{E}_{N},x\neq \theta_{\E_N},$ and taking $a_{k}:=\Delta x(k-1)$, $b_{k}:=x(k-1)$, $m:=N$ in (\ref{summation_by_parts}) we get 
\begin{equation*}
\Vert x\Vert _{N}\Vert x\Vert _{\Delta _{N}}\leqslant N\Vert x\Vert _{\Delta_{N}}^{2}=N\sum\limits_{k=1}^{N}|\Delta
x(k-1)|^{2}=-N\sum\limits_{k=1}^{N-1}\Delta ^{2}x(k-1)x(k)\leqslant N\Vert x\Vert _{N}\Vert x\Vert _{\mathbb{E}_{N}}.
\end{equation*}
Dividing by $\Vert x\Vert _{N}$ we see that $\Vert x\Vert _{\Delta_{N}}\leqslant N \Vert x\Vert _{\mathbb{E}_N}$. Observe that 
\begin{align*}
\Vert x\Vert _{\mathbb{E_{N}}}& =\left( \sum\limits_{k=1}^{N-1}|\Delta x(k)-\Delta x(k-1)|^{2}\right) ^{\frac{1}{2}} \\
& \leqslant \left( \sum\limits_{k=1}^{N-1}|\Delta x(k)|^{2}\right) ^{\frac{1}{2}}+\left( \sum\limits_{k=1}^{N-1}|\Delta x(k-1)|^{2}\right) ^{\frac{1}{2}}\leqslant 2\Vert x\Vert _{\Delta _{N}}.
\end{align*}
This finishes the proof.
\end{proof}
Now we consider (\ref{discretization}) where $v:[0,1]\rightarrow \mathbb{R}$ is a fixed continuous function. Same as before we start with formulating the main theorem.
\begin{tw}
\label{theorem_solvability_of_discretization} Let $v:[0,1]\rightarrow \mathbb{R}$ be a fixed continuous function. Suppose that $f:[0,1]\times\mathbb{R}\rightarrow\mathbb{R}$ satisfies the following assumptions
\begin{itemize}
\item[\textbf{C(c)}:] $f$ is continuous on $[0,1]\times\R$ and it has continuous partial derivative $f_{x}$ with respect to the second coordinate on $[0,1]\times\R$, 
\item[\textbf{D($f$)}:] there exist a positive constants $A$ and $B$, $A<1$, such that for every $t\in[0,1]$ there is $|f(t,x)|\leqslant A|x|+B$,
\item[\textbf{D($f_x$)}:] $\inf\limits_{[0,1]\times\mathbb{R}} f_{x}> -1$.
\end{itemize}
Then problem (\ref{discretization}) has a unique solution.
\end{tw}
Here we present some examples of functions satisfying conditions given above.
\begin{ex}
Let $f_{1},f_{2},f_{3}:[0,1]\times \mathbb{R}\rightarrow \mathbb{R}$ be given by the following formulas
\begin{itemize}
\item $f_1(t,x)=\dfrac{t+\sin(x)}{2x^2+4},$
\item $f_2(t,x)=x e^{t-\pi}-\arctan(x)+e^t,$
\item $f_3(t,x)=\dfrac{x^3+x^2-x}{2x^2+5}+t^3-\sin(t).$
\end{itemize}
\end{ex}
To reach our main result we use operator $D_N:\left(\mathbb{E}_N,\|\cdot\|_{\mathbb{E}_N}\right)\rightarrow \left(\mathbb{E}_N,\|\cdot\|_N\right)$
defined by 
\begin{equation*}
(D_Nx)(k):=\left\{ 
\begin{array}{ll}
\Delta ^{2}x(k-1)-\frac{1}{N^{2}}f\left( \frac{k}{N},x(k)\right) & \text{ for }k=1,2,...,N-1, \\ 
0 & \text{ for }k=0,k=N.
\end{array}
\right.
\end{equation*}
\subsection{Proof of the main result }
Obviously, if we prove that $D_{N}$ is diffeomorphism we also prove Theorem \ref{theorem_solvability_of_discretization}. We employ similar scheme as before. Therefore, firstly we show that $D_{N}$ is $C^{1}$.
\begin{lm}
\label{lemma_D_in_C^1} Assume that \textbf{\textit{C(c)}} is satisfied. Then operator $D_N$ is $C^1$. Moreover, its directional derivative is given by following formula 
\begin{equation*}
D_N^{\prime }(x;h)(k)=\left\{
\begin{array}{ll}
\Delta ^{2}h(k-1)-\frac{1}{N^{2}}f_x\left( \tfrac{k}{N},x(k)\right) h(k) & \text{ for }k=1,2,...,N-1, \\ 
0 & \text{ for }k=0,N.
\end{array}
\right.
\end{equation*}
\end{lm}
\begin{proof}
Fix $x,h\in \mathbb{E}_{N}$. Calculating $\left.\dfrac{d}{d\tau }D_{N}(x+\tau h)\right|_{\tau=0}$ we obtain formula given in assertion. Now let $x_{n}\xrightarrow{\E_N}x_{0}$. Hence 
\begin{align*}
\Vert D_{N}^{\prime }x_{n}-D_{N}^{\prime }x_{0}\Vert _{\mathcal{L}\left(\mathbb{E}_{N},\mathbb{E}_{N}\right) }& =\sup_{\Vert h\Vert _{\mathbb{E}_{N}}=1}\left( \sum_{k=1}^{N-1}\left\vert \tfrac{1}{N^{2}}f_{x}\left( \tfrac{k}{N},x_{0}(k)\right) h(k)-\tfrac{1}{N^{2}}f_{x}\left( \tfrac{k}{N},x_{n}(k)\right) h(k)\right\vert ^{2}\right) ^{\frac{1}{2}} \\
& \leqslant \max_{k\in \mathbb{N}(0,N)}\left\vert f_{x}\left( \tfrac{k}{N},x_{0}(k)\right) -f_{x}\left( \tfrac{k}{N},x_{n}(k)\right)\right\vert\sup_{\Vert h\Vert _{\mathbb{E}_{N}}=1}\tfrac{1}{N^{2}}\Vert h\Vert _{N} \\
& \leqslant\max_{k\in \mathbb{N}(0,N)}\left\vert f_{x}\left( \tfrac{k}{N},x_{0}(k)\right) -f_{x}\left( \tfrac{k}{N},x_{n}(k)\right) \right\vert\rightarrow 0\text{ as }n\rightarrow \infty ,
\end{align*}
where the above relation holds since we have that $x_{n}(k)\rightarrow x_{0}(k)$ for every $k=0,1,...,N$ and since $f_{x}$ is continuous.
\end{proof}
Analogously as with the continuous case, we have to verify the assumptions of Theorem \ref{theorem_Hadamard}.
\begin{lm}
\label{lemma_D'_is_invertible} Assume that \textbf{\textit{D($f_{x}$)}} is satisfied. Then operator $D'_{N}x$ is invertible at
every $x\in\mathbb{E}_{N}$.
\end{lm}
\begin{proof}
Let us fix $a,x\in \mathbb{E}_{N}$ and define a functional $\Phi _{N}:\left(\mathbb{E}_{N},\Vert \cdot \Vert _{\Delta _{N}}\right) \rightarrow \mathbb{R}$ given by 
\begin{equation*}
\Phi _{N}(h):=\tfrac{1}{2}\sum\limits_{k=1}^{N}|\Delta h(k-1)|^{2}+\tfrac{1}{2N^{2}}\sum\limits_{k=1}^{N-1}f_{x}\left( \tfrac{k}{N},x(k)\right)|h(k)|^{2}+\sum\limits_{k=1}^{N-1}h(k)a(k).
\end{equation*}
Observe that $\Phi _{N}$ is $C^{1}$ functional. Using our assumption we obtain existence of some negative constant $B>-1$ such that $\inf\limits_{[0,1]\times \mathbb{R}}f_x\geqslant B$. Hence 
\begin{align*}
\Phi _{N}(h)& =\tfrac{1}{2}\sum\limits_{k=1}^{N}|\Delta h(k-1)|^{2}+\tfrac{1}{2N^{2}}\sum\limits_{k=1}^{N-1}f_{x}\left( \tfrac{k}{N},x(k)\right)|h(k)|^{2}+\sum\limits_{k=1}^{N-1}h(k)a(k) \\
& \geqslant \tfrac{1}{2}\sum\limits_{k=1}^{N}|\Delta h(k-1)|^{2}+\tfrac{B}{2N^{2}}\sum\limits_{k=1}^{N-1}|h(k)|^{2}+\sum\limits_{k=1}^{N-1}h(k)a(k) \\
& \geqslant \tfrac{1+B}{2}\Vert h\Vert _{\Delta _{N}}^{2}-N\Vert a\Vert_{N}\Vert h\Vert _{\Delta _{N}}.
\end{align*}
Hence $\Phi _{N}$ is coercive. Now we show that $\Phi _{N}$ is strictly convex as a sum of linear function $h\mapsto\sum\limits_{k=1}^{N-1}h(k)a(k)$ and strictly convex $\Phi_{N_{1}}(h):=\sum\limits_{k=1}^{N}|\Delta h(k-1)|^{2}+\frac{1}{N^{2}}\sum\limits_{k=1}^{N-1}f_{x}\left( \frac{k}{N},x(k)\right) |h(k)|^{2}$ multiplied by positive integer. Fix $\lambda \in (0,1)$ and take $u,w\in\mathbb{E}_{N}$ such that $u\neq w$. Then 
\begin{align*}
(1-& \lambda )\Phi _{N_{1}}(u)+\lambda \Phi _{N_{1}}(w)-\Phi _{N_{1}}\left((1-\lambda )u+\lambda w\right)  \\
& =\left( \lambda -\lambda ^{2}\right) \sum\limits_{k=1}^{N}|\Delta u(k-1)-\Delta w(k-1)|^{2}+\tfrac{\lambda -\lambda ^{2}}{N^{2}}\sum\limits_{k=1}^{N-1}f_{x}\left( \tfrac{k}{N},x(k)\right) |u(k)-w(k)|^{2}\\
& \geqslant \left( \lambda -\lambda ^{2}\right) \Vert u-w\Vert^2_{\Delta_{N}}+\tfrac{B\left( \lambda -\lambda ^{2}\right) }{N^{2}}\Vert u-w\Vert^2_{N}\geqslant \left( \lambda -\lambda ^{2}\right) (1+B)\Vert u-w\Vert^2_{\Delta _{N}}>0.
\end{align*}
Since $\Phi _{N}$ is coercive and strictly convex, it possess only one critical point, which means that there is only one $h^{\star }\in\E_N$ such that 
\begin{equation*}
\sum\limits_{k=1}^{N-1}\Delta ^{2}h^{\star}(k-1)\psi (k)=\tfrac{1}{N^{2}}\sum\limits_{k=1}^{N-1}f_{x}\left( \tfrac{k}{N},x(k)\right) h^{\star }(k)\psi (k)+\sum\limits_{k=1}^{N-1}a(k)\psi (k)
\end{equation*}
for every $\psi \in \mathbb{E}_{N}$. Therefore there is a unique $h^{\star }$ such that $(D_{N}^{\prime}x)(h^\star)=a$. Since $a$ and $x$ were taken arbitrarily, we get assertion.
\end{proof}
\begin{lm}
\label{lemma_D_is_coercive} Assume that \textbf{\textit{D($f$)}} is satisfied. Then functional $\mathbb{E}\ni x\mapsto \Vert D_{N}x\Vert_N $ is coercive.
\end{lm}
\begin{proof}
Fix $x\in\mathbb{E}_N$. Using assumption we obtain
\begin{align*}
\Vert D_{N}x\Vert_N & \geqslant \left( \sum\limits_{k=1}^{N-1}\left\vert\Delta ^{2}x(k-1)-\tfrac{1}{N^{2}}f\left( \tfrac{k}{N},x(k)\right)\right\vert ^{2}\right) ^{\frac{1}{2}} \\
& \geqslant \left( \sum\limits_{k=1}^{N-1}\left\vert \Delta^{2}x(k-1)\right\vert ^{2}\right) ^{\frac{1}{2}}-\tfrac{1}{N^2}\left( \sum\limits_{k=1}^{N-1}\left\vert Ax(k)+B\right\vert ^{2}\right) ^{\frac{1}{2}}\\
&\geqslant \Vert x\Vert _{\mathbb{E}}-\tfrac{A}{N^{2}}\Vert x\Vert_N -\tfrac{B}{N^{\frac{3}{2}}}\geqslant (1-A)\Vert x\Vert _{\mathbb{E}}-\tfrac{B}{N^{\frac{3}{2}}}.
\end{align*}
Hence $x\mapsto \|D_Nx\|_N$ is coercive.
\end{proof}
Finally we obtain following theorem
\begin{tw}
Assume that \textbf{\textit{C(c)}}, \textbf{\textit{D($f$)}} and \textbf{\textit{D($f_x$)}} are satisfied. Then operator $D_N$ is diffeomorphism.
\end{tw}
\begin{proof}
Using Theorem \ref{theorem_Hadamard}, Lemmas \ref{lemma_D_in_C^1}, \ref{lemma_D'_is_invertible} and \ref{lemma_D_is_coercive} we get assertion.
\end{proof}
Theorem formulated above proves Theorem \ref{theorem_solvability_of_discretization} and consequently (\ref{discretization}) has a unique solution.
\subsection{The convergence results}
Let $x_{N}$ denote solution of problem (\ref{discretization}) for any fixed $N\in \mathbb{N}$, $N\geqslant 2$, and let $x^{\star }$ denotes solution of problem (\ref{continuous_problem}). In this section we will show that sequence $(x_{N})_{N\in \mathbb{N}}$ converge to solution of (\ref{continuous_problem}). We will do this using results from \cite{gaines} and \cite{kellypeterson}. Therefore our main goal is to prove following theorem.
\begin{tw}
\label{theorem_non-spurious_solutions} Assume that \textbf{\textit{C(c)}}, \textbf{\textit{D($f$)}}, \textbf{\textit{D($f_{x}$)}} and 
\begin{itemize}
\item \textbf{\textit{D($v$)}}: function $v$ is continuous,
\end{itemize}
are satisfied. Then the sequence $\left(x_{N}\right)_{N\in\N} $ of solutions of problem (\ref{discretization}) converges to the solution $x^{\star }$ of (\ref{continuous_problem}) in the sense described by (\ref{convergence_of_discretizations}).
\end{tw}
There is some remark in order as concerns the assumptions of both discrete and continuous problem.
\begin{rem}
In \cite{gal} it is shown that when the direct method of the calculus of variation is applied then the constants appearing in the assumption of the continuous problem are inherited by the discrete one, i.e. the discrete problem is solvable with genuinely same assumptions. In our case, the situation is different. As seen from what we have already proved in Theorems \ref{theorem_solvability_of_continuous} and \ref{theorem_solvability_of_discretization} the constants differ since we cannot use the whole range of constants appearing in assumptions \textbf{\textit{C($f$)}} and \textbf{\textit{C($f_{x}$)}}. This is the reason why we use somehow new assumption to reach the convergence result.
\end{rem}
Firstly, we will prove the following lemma
\begin{lm}
\label{lemma_unifromly_bounded} Assume that \textbf{\textit{C(c)}}, \textbf{\textit{D($f$)}}, \textbf{\textit{D($f_x$)}} and \textbf{\textit{D($v$)}} are satisfied. Then there exists a positive constant $M$ such that for every $N\in\mathbb{N}$, $N\geqslant 2$, and every $k\in\mathbb{N}(0,N)$ we have 
\begin{equation}  \label{inequality_solutions_points}
|x_N(k)| \leqslant M.
\end{equation}
\end{lm}
\begin{proof}
Observe that for every fixed $N\in\mathbb{N}$ and every $k\in\mathbb{N}(0,N)$ there is 
\begin{align*}
\left(D_Nx_N\right)(k)=\tfrac{1}{N^2}v\left(\tfrac{k}{N}\right).
\end{align*}
Moreover, since $v$ is continuous, we have 
\begin{align*}
\|D_Nx_N\|_N=\left(\sum_{k=1}^{N-1}\left|(D_Nx_N)(k)\right|^2\right)^\frac{1}{2}=\left(\sum_{k=1}^{N-1}\left|\tfrac{1}{N^{2}}v\left(\tfrac{k}{N}\right)\right|^2\right)^\frac{1}{2} \leqslant \tfrac{\|v\|_\infty}{N^2}\left(\sum_{k=1}^{N-1}1^2\right)^\frac{1}{2}\leqslant\tfrac{\|v\|_\infty}{N^{\frac{3}{2}}}.
\end{align*}
Taking notation and calculations from proof of Lemma \ref{lemma_D_is_coercive} and using Lemma \ref{lemma_inequalieties} we have 
\begin{align*}
M:=\frac{\|v\|_\infty+B}{1-A}\geqslant N^{\frac{3}{2}}\|x_N\|_{\mathbb{E}_N}\geqslant\|x_N\|_{\infty_N}.
\end{align*}
Since $N$ was taken arbitrary, we get assertion.
\end{proof}
In order to prove the main result we will use Lemmas 9.2 and 9.3 from \cite{kellypeterson}.
\begin{proof}[Proof of Theorem \ref{theorem_non-spurious_solutions}]
By Theorem \ref{theorem_solvability_of_discretization} we see that problem (\ref{discretization}) has exactly one solution. Moreover, problem (\ref{continuous_problem}) has exactly one solution by Theorem \ref{theorem_solvability_of_continuous}. By Lemma \ref{lemma_unifromly_bounded} and Lemma 9.3 from \cite{kellypeterson}, there exist positive constants $M$ and $Q$ such that the following estimations hold
\begin{align*}
|x_N(k)|\leqslant M,\hspace{2.5cm} N|\Delta x_N(k-1)|\leqslant Q
\end{align*}
for every $N\in\N$, $N\geqslant 2$, and every $k\in\N(1,N)$. Finally, using Lemma 9.2 from \cite{kellypeterson} we obtain that the discrete problem (\ref{discretization}) has a non-spurious solution in the sense described in the introduction.
\end{proof}

\end{document}